\documentclass[letterpaper, 10 pt, conference]{ieeeconf}  

\IEEEoverridecommandlockouts                              
\usepackage{tikz}
\usepackage{subcaption}
\usepackage{booktabs}
\usepackage{url}
\usepackage{cite}
\usepackage{amsmath,amssymb,amsfonts}
\usepackage{algorithmic}
\usepackage{graphicx}
\usepackage{textcomp}
\usepackage{xcolor}
\def\BibTeX{{\rm B\kern-.05em{\sc i\kern-.025em b}\kern-.08em
    T\kern-.1667em\lower.7ex\hbox{E}\kern-.125emX}}
\newtheorem{thm}{Theorem }
\newtheorem{prop}{Proposition}
\newtheorem{cor}{Corollary}
\newtheorem{lem}{Lemma }
\newtheorem{defn}{Definition}
\newtheorem{rem}{Remark }
\newtheorem{ass}{Assumption}
\newtheorem{ex}{Example}

\usepackage[ruled,vlined,linesnumbered]{algorithm2e}
\usepackage{makecell}
\usepackage{dsfont}

\newcommand{\Real}{\mathbb R}
\newcommand{\B}{\mathbb B}
\newcommand{\eps}{\varepsilon}

\newcommand{\norm}[1]{\left\Vert#1\right\Vert}
\newcommand{\ra}{\rightarrow}
\newcommand{\set}[1]{\left\{#1\right\}}

\newcommand{\D}{\mathcal{D}}
\newcommand{\U}{\mathcal{U}}

\newcommand{\R}{\mathcal{R}}

\renewcommand{\S}{\mathcal{S}}

\newcommand{\Sd}{\mathcal{S}_\delta}

\newcommand{\Ds}{\mathcal{D}_{\operatorname{s}}}
\newcommand{\Dd}{\mathcal{D}_{\operatorname{0}}}

\newcommand{\dx}{\dot{x}}
\usepackage{xcolor}
\newcommand{\ymmark}[1]{{\color{blue} #1}}

\newtheorem{eg}[thm]{\bf Example}
\newcommand{\qed}{\hfill $\square$}

\title{\LARGE \bf Characterization of Safe Stabilization and Control Lyapunov–Barrier Functions via Zubov Equation Formulation}

\author{Yiming Meng and Jun Liu 
\thanks{This research was supported in part by an NSERC Discover Grant and the Canada Research Chairs program. The authors also acknowledge the Mathematics Faculty Computing Facility (MFCF) at the University of Waterloo for computing support.}
\thanks{Yiming Meng is with the AI Thrust, Information Hub, Hong Kong
University of Science and Technology (Guangzhou), China.   Email: \texttt{yimingmeng@hkust-gz.edu.cn}}
\thanks{Jun Liu is with the Department of Applied Mathematics, Faculty of Mathematics, University of Waterloo, Waterloo, Ontario N2L 3G1, Canada.  Email: \texttt{j.liu@uwaterloo.ca}}
}

\begin{document}

\maketitle
\thispagestyle{empty}
\pagestyle{empty}

\begin{abstract}
Design and analysis of stabilizing controllers with safety guarantees for nonlinear systems have received considerable attention in recent years. Control Lyapunov–barrier functions (CLBFs) provide a powerful framework for simultaneously ensuring stability and safety; however, their construction for nonlinear systems remains challenging. To address this issue, we build on recent advances in PDE-based characterizations of control Lyapunov functions and Lyapunov–barrier functions for autonomous systems, and propose a succinct Zubov–HJB PDE formulation for safe stabilization of nonlinear control-affine systems under a common compatibility assumption. We further show that the viscosity solution of this PDE yields a maximal CLBF, enabling (not necessarily continuous) feedback synthesis with stability and safety guarantees. In light of recent advances in neural-network-based methods for solving Zubov-type PDEs, this theoretical framework also provides a natural interface to emerging numerical approaches.
\end{abstract}

\section{Introduction}

Control synthesis for stabilization with safety remains a fundamental and challenging problem in modern control theory and has received increasing attention in a variety of safety-critical applications, such as robotic motion planning and trajectory regulation \cite{fribourg2013control, faulwasser2009model,nilsson2017augmented}. 

Control Lyapunov functions (CLFs) provide a powerful tool for analyzing asymptotic controllability and stabilization, supported by classical converse results and constructive feedback designs for control-affine systems. Similarly, control barrier functions (CBFs) are used to enforce safety constraints \cite{ames2019control, agrawal2017discrete,ames2016control,hsu2015control,nguyen2020dynamic, 9483376}. To achieve stabilization and safety simultaneously, optimization-based approaches, most notably quadratic programming (QP) frameworks, have been proposed to combine Lyapunov and barrier certificates for the real-time synthesis of control inputs. However, these methods do not always provide guarantees, due to potential conflicts between the stabilization objective and safety constraints.

The main challenge lies in ensuring compatibility between the CLF and CBF conditions so that the optimization problem remains feasible. When such compatibility is absent, stability is typically relaxed and treated as a soft constraint. An alternative approach is to merge a CLF and a CBF into a single control Lyapunov–barrier function (CLBF) \cite{romdlony2016stabilization, meng2022smooth}, which serves as a unified Lyapunov-like certificate. When such a function is continuously differentiable, Sontag’s universal formula can be applied to construct a stabilizing feedback that also enforces safety \cite{sontag1989universal}.

Under the compatibility assumption, recent works have studied the construction of continuously differentiable CLBFs \cite{ong2019universal, mestres2022optimization, quartz2025converse, mestres2025converse}. However, such constructions require several stringent conditions. We summarize a hierarchy of underlying causes that explain why continuous differentiability is restrictive. (1) Due to the equivalence between a single smooth CLBF and a continuously differentiable CLF-CBF pair, the existence of a smooth CLBF necessarily requires a continuously differentiable CLF \cite{sontag1999stability}. This, in turn, requires the existence of a continuous stabilizing feedback law, which is generally not guaranteed. (2) Even when a smooth CLF-CBF pair exists, continuous feedback for stabilization with safety may fail to exist in the presence of bounded obstacles. In particular, as shown in \cite{braun2017existence}, topological obstructions can arise. (3) Even when the boundedness of the obstacle is no longer an issue, the geometric structure of the safe set and the topology of the system flow may still prevent the existence of a continuous stabilizing feedback. These observations motivate the need for a more general, possibly nonsmooth, formulation of CLBFs.

In addition to the difficulty of ensuring differentiability of CLBFs (or equivalently, a CLF–CBF pair), the challenge of constructing a CLF defined on a nonlocal domain raises further concerns that the resulting feedback may guarantee only local stabilization with safety. Motivated by this issue, in \cite{li2025solving}, a discounted stabilize–avoid value function is proposed, which certifies safety at the set level but does not induce a Lyapunov–barrier feedback law guaranteeing finite-time reachability and safety. The work in \cite{grune2015zubov} adopts a PDE-based formulation but relies on a nontrivial terminal cost, resulting in less convenient boundary conditions and a Hamilton–Jacobi equation that is not readily solvable with existing computational tools. The work in \cite{quartz2025converse} provides an explicit formula for a CLBF; however, it necessarily relies on two functions that solve two HJ-type PDEs. Recent work \cite{meng2025towards} has developed a Zubov PDE   formulation for constructing a unified Lyapunov–barrier function (for autonomous systems) by solving a Dirichlet boundary value problem associated with the Zubov equation, defined via a proper indicator function of the obstacle set. The resulting Lyapunov–barrier function is defined on the maximal domain where the system admits stabilization under safety constraints. Moreover, it has been shown to be a maximal Lyapunov function for a suitably modified dynamical system, which enables the use of existing tools, such as \cite{liu2024lyznet, liu2025physics}, to learn near-maximal Lyapunov–barrier functions. 
Inspired by this idea, we propose a Zubov–HJB formulation for control-affine systems that yields a (not necessarily smooth) CLBF. The construction of CLFs and CBFs is reduced to a single PDE, whose solvability is equivalent to the existence of a continuous CLBF and whose solution behaves as a CLF on the safe null-controllability domain.

A recent renaissance in Hamilton–Jacobi (HJ) type PDE methods has provided a promising alternative for constructing Lyapunov and control Lyapunov functions \cite{camilli2008control, grune2023examples, liu2025formally, meng2024physics}. 
In particular, \cite{liu2025formally} has investigated Zubov-type equations for control systems using physics-informed neural networks (PINNs) to compute CLFs. This approach enables formal verification via SMT solvers and has been shown to outperform traditional SOS and rational CLF methods, while explicitly incorporating the Zubov-HJB PDE structure. In this view, the proposed framework also serves as an important interface with emerging numerical methods.
 
\textbf{Notation:} We list some notation used in this paper. 
For $x\in\Real^n$ and $r\ge 0$, we denote the ball of radius $r$ centered at $x$ by $\B(x;r)=\set{y\in\Real^n:\,\norm{y-x}\le r}$, where $\norm{\cdot}$ is the Euclidean norm. For a closed set $A\subset\Real^n$ and $x\in\Real^n$, we denote the distance from $x$ to $A$ by $\norm{x}_{A}=\inf_{y\in A}\norm{x-y}$. 
For a set $A\subseteq\Real^n$, $\overline{A}$ denotes its closure.  
For two sets $A,B\in\Real^n$, we use $A\setminus B$ to denote the set difference defined by $A\setminus B=\set{x:\,x\in A,\,x\not\in B}$.  We also write $a\gtrsim b$ if there exists a $C>0$ (independent of $a$ and $b$) such that $a\geq Cb$.

\section{Preliminaries}\label{sec:prelim}
\subsection{Dynamical Systems and Stabilization-with-Safety}\label{sec: pre}
Consider a continuous-time dynamical system 
\begin{equation}\label{eq:sys}
     \dx  = f(x) + g(x)u
\end{equation}
where $f:\,\Real^{n}\to\Real^{n}$ and $g:\,\Real^{n}\to \Real^{n\times m}$ are assumed to be locally Lipschitz; $x\in\Real^n$ is the state, $u\in \Real^m$ is the control input.  For simplicity, in the context of controlled trajectory, we overload the notation $u$ as the control signal,   i.e. $u:[0,\infty)\ra \Real^m$; the unique solution starting from $x_0$ under the control $u$ is denoted by $\phi(t; x_0, u)$. We may also write the solution as $\phi(t)$ or $\phi$ if the omitted arguments are not emphasized. We are interested in the case where the maximal interval of existence is $[0,\infty)$ for admissible controls $u\in\U:=L^\infty([0, \infty);\Real^m)$.

Without loss of generality, we also assume $f(0) = 0$. Denote $\mathcal{X}_0\subseteq\Real^n$ as the set of initial conditions and $U\subseteq\Real^n$ as the obstacle. Throughout the paper, we also denote $S=U^c$ for simplicity. We then  consider safe stabilization w.r.t.  the origin as follows. 
\begin{defn}[Stabilization-with-safety guarantee]
We say that \eqref{eq:sys} satisfies a stability-with-safety specification $(\mathcal{X}_0,U)$ if, for all $x \in \mathcal{X}_0$, there exists an admissible control signal $u\in\U$ such that $\lim_{t\ra\infty}\|\phi(t;x,u)\|\ra 0$ and $\phi(t;x,u) \notin U$ for all $t \ge 0$. \qed
\end{defn}

Analogously to the domain of null-controllability, which is defined as
\begin{equation*}
    \D_0:= \set{x\in\Real^n: \exists u\in\U \text{ s.t. } \lim_{t\ra\infty}\|\phi(t;x,u)\|\ra 0}, 
\end{equation*}
we define the domain of null-controllability with safety constraint   as
\begin{equation*}
    \begin{split}
        \D:= &
\left\{
x\in\mathbb{R}^n :
\exists u\in\mathcal U \text{ s.t. }
\lim_{t\ra\infty}\|\phi(t;x,u)\|\ra 0 \right.\\
&  \left. \text{ and }
\phi(t;x,u)\notin U \ \forall t\ge0
\right\}.
    \end{split}
\end{equation*}

To this end,  we assume   local stabilizability of \eqref{eq:sys}. 
\begin{ass}\label{ass: local}
    We assume that there exists an open ball $\B(0;r)$, a constant $k>0$, and $\beta\in\mathcal{KL}$ such
that for any $x\in\B(0;r)$ there exists $u\in\U$ with $\|u\|_\infty\leq k$, and $\|\phi(t;x,u)\leq \beta(\|x\|,t)$ for all $t\geq 0$. \qed
\end{ass}

For future reference, we   make the following assumptions on $S$ to facilitate the analysis.
\begin{ass}\label{ass: h}
    We assume 
    \begin{enumerate}
        \item $S = \set{x\in\Real^n: h(x)<1}$ with some locally Lipschitz continuous function $h$; 
        \item $S$ is open and connected;
        \item $\B(0;r)\subseteq S\subseteq\D_0$, where $\B(0;r)$ is the local stabilization region as required in Assumption \ref{ass: local}. \qed
    \end{enumerate}
\end{ass}

We also define the exit time from $S$, or equivalently entry time in $U$, as
\begin{equation}
    \tau(x, u) = \inf\set{t\geq 0: \phi(t;x,u)\in U}, 
\end{equation}
with the convention that $\inf\emptyset = +\infty$. A closely related notion of weak invariance is defined as follows:
\begin{equation}
    \Ds:=\set{x\in\Real^n: \exists u\in\U \text{ s.t. } \tau(x,u) = \infty}. 
\end{equation}

\begin{rem}
    Since, by definition, $\D\subseteq\D_0\cap\D_s$, and 
$\D_s\subseteq S$, it follows that $\D\subseteq \D_0\cap S$. To study the behavior on $\D$, the only relevant part is $\D_0\cap S$. The assumption that $S\subseteq \D$ is adopted for convenience, since in any case we do not consider the regions $\D_0\setminus S$ or $S\setminus\D_0$. We would like to clarify that the common assumptions in \cite{ong2019universal, mestres2025converse, quartz2025converse} all explicitly or implicitly require $S\subseteq\D_0$. \qed
\end{rem}

\subsection{Nonsmooth Value Functions for Asymptotic Controllability and Controlled Forward Invariance}\label{sec: clf_cbf}

Continuously differentiable CLFs and CBFs provide a classical framework for stabilization and safety, with necessary and sufficient existence conditions. Their combination into a CLBF is appealing for enforcing stabilization with safety. However, differentiability imposes restrictive requirements, including continuous feedback and geometric or topological conditions on the safe set, which often fail in practice. In contrast, nonsmooth CLFs and CBFs can characterize asymptotic controllability and safety while permitting discontinuous feedback. We introduce these notions using the concept of the proximal subgradient, defined below.

\begin{defn}
    For any $\varphi\in\mathcal{C}(\Real^n)$ and $x\in\operatorname{dom}(\varphi)$, a vector $\zeta\in\Real^n$  is called a proximal subgradient of 
$\varphi$ at $x$ if there exist constants $\sigma, \eta>0$ such that 
\begin{equation}
    \varphi(y)\geq \varphi(x)+\langle \zeta, y-x\rangle - \sigma \|y-x\|^2, \;\forall y\in\B(x;\eta). 
\end{equation}
The set of all proximal subgradients of 
$\varphi$ at $x$ is denoted by $\partial_P\varphi(x)$. \qed
\end{defn}

Converse theorems state that the system \eqref{eq:sys} is asymptotically controllable if and only if there exists a (not necessarily smooth) CLF $V\in\mathcal{C}(\D_0)$ such that $V$ is positive definite on $\D$, and  there exist a positive definite continuous function $w$ and a  $\sigma\in\mathcal{K}_\infty$ such that for all $x\in\D_0$, 
\begin{equation}\label{E: clf}
    \inf_{\|u\|\leq \sigma(\|x\|)}\langle\zeta, f(x)+g(x)u\rangle \leq -w(x), \;\forall\zeta\in \partial_PV(x). 
\end{equation}

Similar to the generalized notion of CLFs above, we now present a nonsmooth version of CBFs. Let $h$ be the function defining $S$  (as in Assumption \ref{ass: h}). Then $h$ is called a CBF for system \eqref{eq:sys} w.r.t $S$ if there exists a locally Lipschitz class $\mathcal{K}_\infty$ function $\alpha$ such that for all $x\in S$, 
\begin{equation}\label{E: cbf}
    \inf_{u\in\Real^m} \langle \zeta, f(x)+g(x)u\rangle\leq \alpha(1-h(x)), \;\forall\zeta\in \partial_Ph(x). 
\end{equation}

\begin{rem}\label{rem: equiv}
    Based on the connections between proximal subgradients and other generalized derivatives for characterizing nonsmooth regularity \cite{clarke1998nonsmooth}, there are several equivalent formulations of nonsmooth CLFs and CBFs. For instance, one may use the lower right Dini directional derivative to replace the left-hand side of \eqref{E: clf} and \eqref{E: cbf} (e.g., the complete CLF concept in \cite{braun2017existence}), or adopt subdifferential and viscosity supersolution formulations (e.g., Eq.~(1.2) in \cite{camilli2008control}).
   We adopt proximal conditions, as they provide a natural geometric perspective at nondifferentiable points, are linear in $f(x)+g(x)u$, and play an important role in constructing discontinuous feedback \cite{sontag1999stability}. Note that all of the above formulations reduce to the classical  notion when the value functions are   differentiable. \qed
\end{rem}

\subsection{Problem Formulation}\label{sec: Prob}
We aim to: (i)
Characterize a CLF V defined on the entire
domain $\D$ via a Lyapunov-like equation formulation; (ii)
Construct a CLBF $W$ such that $W$ is
positive definite and satisfies
\begin{enumerate}
    \item $\inf_{u\in\Real^m}\langle\zeta, f(x)+g(x)u\rangle<0$ for all $x\in\D\setminus\{0\}$ and for all $\zeta\in\partial_PW(x)$;
    \item $W(x)<1$ for all $x\in\D$ and $W\geq 1$ for all $x\in U$. 
\end{enumerate}

\section{Fundamental Properties   of Domain of Null-controllability with Safety
Constraint}
Before characterizing stabilization-with-safety guarantees and deriving the converse Lyapunov–barrier function, we use this short section to present basic properties of
 $\D$. Recall the related definitions and assumptions in Section \ref{sec: pre}.  

\begin{prop}
   Under Assumption \ref{ass: local} and \ref{ass: h},  $\D$ is open and connected. 
\end{prop}
\begin{proof}
By Assumption \ref{ass: local}, there exists some $r>0$ such that $\B(0;r)\subseteq\D$, and every point in $\B(0;r)$ can be steered to the origin.  We aim to show that, for any $x\in \D$, there exists $\rho>0$ such that $\B(x;\rho)\subset\D$.
    Take any $x\in\D$. By definition, there exists $u_x\in\U$ such that $\phi(t;x,u_x)\in S$ for all $t\geq 0$,  and there exists a $T_x\in(0, \infty)$ such that $\phi(T_x;x,u_x)\in\B(0;r/2)$. Note that the set $\set{\phi(t;x,u_x), 0\leq t\leq T_x}$ is compact, and there exists $\delta_x>0$ such that $\|\phi(t;x,u_x)\|_U\geq \delta_x$ for all $t\in[0, T_x]$. 
 
  Now, let $u_z\in\U$ be such that $u_z(t) = u_x(t)$ for all $t\in[0, T_x]$. By the continuity of the trajectory, there exists a $\rho>0$ such that whenever $z\in\B(x;\rho)$, $\|\phi(t;x,u_x)-\phi(t;z,u_z)\|<\min\{\delta_x/2, r/2\}$ for all $t\in[0, T_x]$. We then have $\|\phi(t;z,u_z)\|_U\geq \|\phi(t;x,u_x)\|_U-\|\phi(t;x,u_x)-\phi(t;z,u_z)\|\geq \delta_x/2$ for all $t\in[0, T_x]$. On the other hand, we notice that $\|\phi(T_x;x,u_x)-\phi(T_x;z,u_z)\|<r/2$, which implies $\phi(T_x;z,u_z)\in \B(0;r)$, and therefore $z\in\D$. The openness of $\D$ thus follows.

  Since $\B(0;r)\subseteq \D$ is connected and every $x\in\D$ can be joined to $\B(0;r)$ by a trajectory lying entirely in $\D$, it follows that $\D$ is path-connected, and hence connected.
\end{proof}

\begin{prop}\label{prop: boundary}
    For any $y\in\partial \D$, it must fall into one of the following three cases: (1) $y\in\partial\D_0$; (2) $y\in \partial \Ds \cap \D_0$; (3) $y\in (\Ds\cap\D_0)\setminus \D$. 
\end{prop}
\begin{proof}
    It is clear that $\D\subseteq\D_0\cap\D_s$. Therefore, for any point $y\in\partial\D$, it must be that either (1)  $y\notin \D_0$, or (2) $y\in\D_0\setminus \D_s$, or (3) $y\in(\D_0\cap\D_s)\setminus\D$. Then (3) in the statement   follows. For~(1), we have $\B(y;r)\cap\D\neq \emptyset$ for all $r>0$. Then, $\B(y;r)\cap\D_0\neq \emptyset$ for all $r>0$ and $y\in\overline{\D_0}$. Since $\D_0$ is open \cite{camilli2008control}, we have $y\in\overline{\D_0}\setminus\D_0=\partial\D_0$. For~(2), it falls in a similar manner as (1). Since $y\in\partial\D$ and $y\notin \D_s$,  every neighborhood of  $y$ intersects with $\D_s$ and $\D_s^c$, which implies $y\in\partial\D_s$. Together with $y\in\D_0$, we conclude that $y\in\D_0\cap\partial\D_s$. 
\end{proof}

    Note that for (3) of Proposition~\ref{prop: boundary}, stabilization and safety are individually achievable but never simultaneously by the same control.

\section{Converse Lyapunov-Barrier Function for Stability with Safety Guarantees} \label{sec:stability}

In this section, we  derive a set of PDE-based converse CLBF theorems for \eqref{eq:sys}, which characterizes a stabilization-with-safety guarantee specification $(\mathcal{X}_0, U)$ for some unsafe set $U\subset\Real^n$. 

\subsection{Key Concepts and Assumptions}
We first impose a restriction on $\Ds$ and $\D_0$ to eliminate geometric and dynamical obstructions to safe convergence to the origin.
\begin{ass}\label{ass: comp}
    Let $v:\D_0\ra\Real$ be a  CLF as defined in Section \ref{sec: clf_cbf}. We assume that  $\Gamma(x)\neq\emptyset$ for every $x\in \{x: 1-\eps\leq h(x)<1\}$ for some $\eps>0$, where $h$ is as in Assumption \ref{ass: h}, $\Gamma(x)$ is defined in \eqref{E: comp}, and $w_\epsilon$ in \eqref{E: comp} is continuous and positive. \qed
\end{ass}

 Assumption \ref{ass: comp} is weaker than that in \cite{mestres2022optimization, ong2019universal} in terms of the region where the inequalities hold, and the safe set in our paper is open, with a more general nonsmooth formulation. Another version of the compatibility assumption can be found in \cite{quartz2025converse}, which considers a compact safe set. 

We have the following immediate result.  
\begin{cor}\label{cor: boundary}
    Given Assumption \ref{ass: comp}, for any $y\in\partial\D$, $y\in\partial\D_0$ or $y\in\partial S\cap \D_0$.   \qed
\end{cor}
\begin{proof}
   Recall Proposition \ref{prop: boundary}.  It can be verified that $\D_s \subseteq S \subseteq \D_0$. Moreover, by the properties established in the Appendix, the compatibility condition eliminates the possibility that stabilization and safety are individually achievable but not simultaneously achievable by the same control.  Hence, $y\notin(\D_s\cap\D_0)\setminus\D$.  Additionally, Assumption \ref{ass: comp} also implies that $\D_s\cap\D_0 = S\cap \D_0$, and since $\D_0$ is open, $y\in\partial\D_s\cap\D_0\implies y\in\partial S\cap \D_0$. 
\end{proof}

    The compatibility assumption is standard in the current literature and is sufficient for the following analysis of stabilization with safety. It is also a reasonable assumption in the sense that, to the best of the authors’ knowledge, although its necessity remains unproven, no counterexample has been constructed to date. We therefore adopt this assumption.

 To proceed, we introduce the following  controlled version of a proper indicator defined on an open domain, which generalizes the idea of a radially unbounded function on   $\Real^n$.

\begin{defn}
    Let $\{0\}$ be contained in an open set $D\subseteq\Real^n$. A continuous function $\eta:\,D\times \Real^m\ra\Real_{\ge 0}$ is said to be a proper controlled indicator for the origin on $D$ if the following two conditions hold: (1) $\eta(x,u)=0$ if and only if $x = 0$ and $u=0$; (2) $\inf_{u\in\Real^m} \eta(x_m,u) \ra \infty$ for any sequence $\set{x_m}$ in $D$ such that either $x_m\ra p\in \partial D$ or $\norm{x_m}\ra\infty$ as $m\ra\infty$; (3) for any compact set $K\subset D$ and any $x\in K$, if $\|u\|\ra\infty$, then $\eta(x,u)\ra\infty$.   \qed
\end{defn}

Similar to \cite{meng2025towards},  the indicator function that serves as a running cost must be carefully designed so that the resulting value function   characterizes the set $\D$. We impose the following assumptions.

\begin{ass}\label{ass: ass_1}
  Given a proper indicator $\eta:S\times\Real^m\ra\Real$ for $\set{0}$ and $S$, we denote $\omega(x) := \inf_{u\in\Real^m}\eta(x,u)$ and assume that
    \begin{enumerate}
    \item For any $\delta>0$, there exists $\eta_\delta>0$ such that $\omega(x) \geq \eta_\delta$ for all $x\in S$ with $\norm{x}\geq\delta$. 
   
        \item  $\eta(x,u)\lesssim(\alpha_2^{-1}(\|x\|))^b$ for all $(x,u)\in\B(0;r)\times\B(0;k)$ for some $b>0$, where $r$ and $k$ are as in Assumption \ref{ass: local}, and   $\alpha_1,\alpha_2\in\mathcal{K}$ satisfy $\beta(r,t)\leq \alpha_2(\alpha_1(r)e^{-t})$. 
        \item  There exists a continuous function $m:S\times\Real^m\ra\Real_{\ge 0}$ and a compact set $K\subset S\times\Real^m$, such that   $m(x,u)\leq C_K$ for all $(x,u)\in K$ and some $C_K>0$, $\eta(x,u)\geq m(x,u)$ for all $(x,u)\in (S\times \Real^m)\setminus K$, and $\|\phi(T;x,u)-x\|\leq \int_0^T m(\phi(s;x,u),u(s))ds$ for all $x\in S$, $u\in\U$, $T\geq 0$. 
         \item There exists an $a$-neighborhood of $\partial S$ and a   function 
$\psi(x):\Real_+\ra\Real_+$, such that $\psi$ is non-increasing, $\psi(a)\ra \infty$ as $a\ra 0$,  $\omega(x)\gtrsim\psi(\norm{x}_{\partial\Ds})$, and $\int_0^a\psi(x)dx = \infty$. 
        \qed
    \end{enumerate}
\end{ass}

Items 1)–3) of Assumption~\ref{ass: ass_1} are analogous to those imposed in \cite{camilli2008control} [16], except for a different choice of $\eta(x,u)$ and a more general formulation. These conditions are imposed to regulate the growth rate near the origin and at large. Item 4) of   Assumption~\ref{ass: ass_1}  is introduced to play a role similar to that in \cite[Assumption 2-2)]{meng2025towards}. It ensures that the integral, with $\eta$ as the integrand, diverges in a neighborhood of $\partial S$, where the integrand itself becomes unbounded, so that the value   remains consistent with the terminal cost $q$ as trajectories approach the boundary of $S$. 

\subsection{Maximal Converse Lyapunov-Barrier Function for Stability-with-Safety Specification}

This subsection   explicitly construct  value functions, which serve as  maximal CLF on $\D$ and CLBF for $(\mathcal{X}_0, U)$. 

Define $q(x) = q_0(x)$ for $x\in S$ and $q(x) = \infty$ for $x\in\partial S\cup(\Real^n\setminus S)$, 
where $q_0\in\mathcal{C}(S)$ and $q_0(x)\ra\infty$ as $x\ra\partial S$. Consider   a proper indicator $\eta: S\times \Real^m\ra \Real$ on $S$. For $u\in\U$, we further define 
\begin{equation}\label{E: V2}
    V(x) = \inf_{u\in\U} J(x,u),
\end{equation}
where $J(x,u)$ is defined in \eqref{E: V}.

\begin{figure*}[!t]
\normalsize

\begin{equation}\label{E: comp}
        \Gamma(x) = \{u\in\Real^m: \langle\zeta, f(x)+g(x)u\rangle <-w_\eps(x) \;\forall \zeta\in \partial_Pv(x) \text{\;and}\;\langle\xi, f(x)+g(x)u\rangle\leq \alpha(1-h(x)) \;\forall \xi\in \partial_Ph(x)\}
    \end{equation}
    \begin{equation}\label{E: V}
   J(x, u)  =
\begin{cases}
\displaystyle \int_0^{\tau(x,u)} \eta(\phi(t;x,u), u(t))\,dt + q(\phi(\tau(x,u);x,u)), & \tau(x,u) < \infty, \\[1.2ex]
\displaystyle \int_0^{\infty} \eta(\phi(t;x,u), u(t))\,dt, & \tau(x,u) = \infty,
\end{cases}
\end{equation}

\begin{equation}\label{E: dpp}
V(x)  
              =   \inf_{u\in\U}\set{\int_0^{t\wedge\tau(x,u)}\eta(\phi(s;x,u), u(s))ds + V(\phi(t\wedge \tau(x,u); x,u))}. 
    \end{equation}

  \begin{equation}\label{E: zubov_in}
   - \sup_{u\in\Real^m}\{DW(x) (f(x)+g(x)u) + (1-W(x))\varphi(W(x))\eta(x, u)\} = 0  
\end{equation}

\hrulefill
\vspace*{4pt}
\end{figure*}

\begin{lem}\label{lem: dpp}
    For all $x\in\Real^n$ and $t\geq 0$, the \textit{Dynamic Programming Principle}  for $V$ in \eqref{E: dpp} holds. 
\end{lem}

\begin{proof}
    For simplicity, we denote $F_t(x,u):=\int_0^{t\wedge\tau(x,u)}\eta(\phi(s;x,u), u(s))ds + V(\phi(t\wedge \tau(x,u); x,u))$ and $y=y(t,x,u) := \phi(t\wedge\tau(x,u);x,u)$. 
    We aim to show   $V(x) = \inf_{u\in\U}F_t(x,u)$ for any $x\in\Real^n$.  For any $u,v\in\U$, denote $u'(s) = u(s)$ for $s\in[0, t]$ and $u'(s) = v(s-t)$ for $s\in(t, \infty)$ 
and 
$u''(s):= u(t+s)$ for all $s\geq 0$. 

Fix any $u\in\U$.  Suppose $\tau(x,u)\geq t$. which implies $\tau(x,u')=t + \tau(y,v)$ and   $\phi(t+s;x,u') = \phi(s;y,v)$ for all $s\geq 0$. Then, $J(x,u') = \int_0^t\eta(\phi(s;x,u),u(s))ds + \int_0^{\tau(y,v)}\eta(\phi(s;y,v),v(s))ds + q(\phi(\tau(y,v);y,v))$ when $\tau(y,v)<\infty$, and $J(x,u') = \int_0^t\eta(\phi(s;x,u),u(s))ds + \int_0^{\infty}\eta(\phi(s;y,v),v(s))ds$ when $\tau(y,v)=\infty$. In either case, we have $ J(x,u') = \int_0^t\eta(\phi(s;x,u),u(s))ds  + J(y,v)$.  
Therefore, 
    \begin{equation}
    \begin{split}
    V(x)  = &\inf_{w\in\U}J(x,w)\\
    \leq &\inf_{v\in\U}J(x,u')\\
     = & \int_0^tl(\phi(s;x,u),u(s))ds  + V(y) = F_t(x,u). 
    \end{split}
    \end{equation}
Since $u\in\U$ is arbitrary, we have $V(x)\leq \inf_{u\in\U}F_t(x,u)$. For $\tau(x,u)<t$, it reduces to   $F_t(x,u) = J(x,u)$, in which case we also have $V(x)\leq \inf_{u\in\U}F_t(x,u)$. 

Now we show the inverse side.  Suppose $\tau(x,u)\geq t$. Then, $\phi(t+s;x,u) = \phi(s;y,u'')$ for all $s\geq 0$ and $\tau(y,u'') + t = \tau(x,u)$. Similar to above, in either case where $\tau(x,u)<\infty$ or $\tau(x,u)=\infty$, we have $J(x,u) = \int_0^t \eta(\phi(s;x,u),u(s))ds + J(y,u'')$. It follows that $J(x,u)\geq \int_0^t \eta(\phi(s;x,u),u(s))ds + V(y) = F_t(x,u)$. Then, for any $\eps>0$, there exists some $u_\eps\in\U$ and $V(x)+\eps\geq J(x,u_\eps)\geq F_t(x,u_\eps)\geq \inf_{u\in\U}F_t(x,u)$. Sending $\eps\downarrow 0$, we have $V(x)\geq \inf_{u\in\U}F_t(x,u)$.  For $\tau(x,u)<t$, it reduces to   $F_t(x,u_\eps) = J(x,u_\eps)$. Following the same argument, we still have $V(x)\geq \inf_{u\in\U}F_t(x,u)$. 
\end{proof}

The following theorem establishes key properties of 
 $V$. 

\begin{thm}\label{thm: V}
  Under Assumption \ref{ass: local}-\ref{ass: ass_1}, the function $V:\Real^n\ra\Real\cup\set{+\infty}$ defined by \eqref{E: V2} satisfies the following properties:
    \begin{enumerate}
        \item $V(x)<\infty$ if and only if $x\in\D$;
        \item $V(x)\ra\infty$ as $x\ra\partial\D$ or $\|x\|\ra \infty$;
        \item $V$ is positive definite on $\D$;
        \item $V$ is continuous on $\D$. 
    \end{enumerate}
\end{thm}

\begin{proof}
    (1) Pick a point $x\in\D$. Then there exists a $u\in\U$ and a $t\in(0, \infty)$ such that $\tau(x,u) = \infty$ and $\phi(t;x,u)\in\B(0;r)$ as stated in Assumption \ref{ass: local}. By Lemma \ref{lem: dpp}, $V(x)\leq \int_0^t \eta(\phi(s;x,u),u(s))ds + V(\phi(t;x,u))$. The finiteness of both terms follows from the argument in  \cite[Proposition 3.3]{camilli2008control}  under Assumption \ref{ass: ass_1}-2). To show the converse, we suppose $x\notin \D$ and denote $\tau_r(x,u) = \inf\set{t\geq 0: \phi(t;x,u)\in\B(0;r)}$, where the $\B(0;r)$ is as stated in Assumption~\ref{ass: local}. Then, for all $u\in\U$, either $\tau(x,u)<\infty$ or $\tau_r(x,u) = \infty$. Both cases imply $J(x,u)=\infty$, the former by definition  and the latter by the same argument as in \cite[Proposition 3.3]{camilli2008control}. Therefore, 
$V(x)=\infty$, which leads to a contradiction.

(2) By 
Corollary~\ref{prop: boundary}, for any $y\in\partial\D$, we have either $y\in\partial\D_0$ or $y\in\partial S\cap\D_0$. Consider $\{x_m\}\subset\D$. The cases $x_m \to \infty$ and $x_m \to y \in \partial \D_0$ follow the same proof as in \cite[Proposition 3.6]{camilli2008control}. It suffices to
show the case where $x_m\ra y\in\partial S \cap \D_0\subseteq\partial\D$. 


Recall that $\omega(x) := \inf_{u\in\Real^m}\eta(x,u)$, as defined in Assumption~\ref{ass: ass_1}. By Lemma \ref{lem: dpp}, for sufficiently large $m$, pick any $\eps>0$, there exists $u_m$ such that, for any $T>0$, 
    \begin{equation*}
\begin{split}
       &  V(x_m) +\eps \\
       \geq & \int_0^{T}\eta(\phi(s;x_m, u_m), u_m(s))ds + V(\phi(T;x_m, u_m))\\
       \geq & \int_0^{T}\eta(\phi(s;x_m, u_m), u_m(s))ds \\
       \geq &  \int_0^{T}\omega(\phi(s;x_m, u_m))ds. 
\end{split}
\end{equation*}
Let $a$ be given as in (4) of Assumption \ref{ass: ass_1}.  For any $t\in (0, t_\delta)$ for some sufficiently small $t_\delta>0$, we have
    \begin{equation}\label{E: eps_m}
    \begin{split}
      &  \|\phi(t;x_m, u_m)\|_{\partial S} \\ \leq &\norm{\phi(t;x_m, u_m)-\phi(\tau(y, u_m);x_m, u_m)} \\
      & + \norm{\phi(\tau(y, u_m);x_m, u_m)-\phi(\tau(y, u_m);y, u_m)}\\
      \leq & \norm{f}t +\eps_m   
      \leq   \norm{f} t_\delta +\eps_m \leq a, 
    \end{split}
\end{equation}
 where $\eps_m:= \|x_m-y\|\ra 0$, and the second inequality follows directly by the locally Lipschitz continuity of $f$ and that $\sup_{u\in\U}\tau(y, u) = 0$.  It follows that 
\begin{equation*}
    \begin{split}
        V(x_m)  \geq &\int_{0}^{\delta_t}\omega(\phi(s;x_m))ds \\\gtrsim  & \int_{0}^{\delta_t}\psi(\|f\|s)+\eps_m)ds 
        \gtrsim  \int_{\eps_m}^{\eps_m+\norm{f}\delta_t}\psi(x)dx
    \end{split}
\end{equation*}
and $\liminf_{m\ra\infty}V(x_m)\gtrsim \int_{0}^{\norm{f}\delta_t}\psi(x)dx$. Hence, $V(x_m)\ra\infty$ by  Assumption \ref{ass: ass_1}-4).  

(3)  It is clear that $V(0)=0$. To show the converse, we assume $x\neq 0$ but $V(x)=0$. Then, there exists $\{u_k\}\subset\U$ such that $J(x,u_k)\ra 0$ as $k\ra\infty$. Let $\delta=\|x\|/2$, and define $t_k:=\inf\{t\geq 0: \|\phi(t;x,u_k)\|\leq \delta\}$. By (1) of Assumption \ref{ass: ass_1}, $J(x, u_k)\geq\int_0^{t_k} \eta(\phi(s;x,u_k), u_k(s))ds\geq \eta_c t_k$, which implies $t_k\ra 0$. Now let $K\subset S\times\Real^m$ be as in (3) of Assumption \ref{ass: ass_1},  consider $E_k:=\{t\in[0, t_k]: \phi(t;x,u_k), u_k(t))\in K$. It follows that 
    \begin{equation*}
    \begin{split}
        J(x,u_k) &\geq \int_{[0, t_k]\setminus E_k}\eta(\phi(s;x,u_k), u_k(s))ds\\
       & \geq  \int_{[0, t_k]\setminus E_k}m (\phi(s;x,u_k), u_k(s))ds\\
       & \geq \|\phi(t_k;x,u_k)-x\|-\int_{E_k}m(\phi(s;x,u_k), u_k(s))ds\\
       & \geq \|x\|-\|\phi(t_k;x,u_k)\|-C_Kt_k\\
    & \geq \delta - C_Kt_k. 
    \end{split}
\end{equation*}
 Therefore, $\liminf_{k\ra\infty} J(x,u_k)\geq \delta>0$, which contradicts $J(x,u_k)\ra 0$. The above shows $V(x)>0$ for $x\neq 0$. The proof is complete. 

(4) The proof follows the same procedure as in \cite{camilli2008control}. We omit it due to repetition.
\end{proof}

    Note that, for any $x_m\ra \partial \D$ such that $V(x_m)\leq C$ for all $m\in\mathbb{N}$, one has $J(x_m, u_m)\leq V(x_k)+\eps\leq C+\eps$. However, by boundedness of $\set{\phi(t;x_m,u_m), t\geq 0, m\in\mathbb{N}}$ and the uniform boundedness of $\int_0^t\|u_m(t)\|dt$ in $m$, one has $\|\phi(t;x_m,u_m)-\phi(t;y,u_m)\|<\delta$ for any $\delta, t>0$ and sufficiently large $m$, which eventually implies that $y\in\D_0$.   This means that the only   case in which $V(x_m)$ does not diverge is when $x_m\ra y\in\D_0$.  This is why condition (4) in Assumption~\ref{ass: ass_1} is necessarily required, to impose an additional rate of divergence near $\partial S$.

\begin{ex}
    We show a counterexample where the integrand diverges near the boundary but $V(x_m)\not\to \infty$ for $x_m\to y\in\partial\D$. Consider $\dot{x} = -x+u$. Let $u'\equiv 0$. Following the same argument as in \cite[Example 2]{meng2025towards}, onn can show that $J(x_m,u')$ is uniformly bounded in $m$. 
  Consequently,  $V(x_m) \leq  J(x_m,u')$ is uniformly bounded.  \qed
\end{ex}

\begin{rem}\label{rem: sup_tau}
    To obtain the second inequality from the first inequality in \eqref{E: eps_m}, it only needs to assume that $\sup_{m} \tau(y,u_m)<\infty$. However, the compatibility assumption directly implies $\sup_{m} \tau(y,u_m)= 0$, which is somewhat stronger than necessary, yet a weaker version of the assumption is nontrivial in the literature to the best of the authors' knowledge. \qed 
\end{rem}

Based on Theorem 1, and by similar reasoning as in \cite[Section. 3]{liu2025physics}, 
one can also show that $V$ in \eqref{E: V} is the unique  viscosity solution to the HJB
\begin{equation}\label{E: lyap}
   - \sup_{u\in\Real^m}\{DV(x) (f(x)+g(x)u)  -\eta(x,u)\} = 0
\end{equation}
with $V(0) = 0$. 
As a direct consequence, such 
$V$ is also a viscosity supersolution on 
$\D$ satisfying \eqref{E: clf}. We omit the proof due to similarity. 

Considering  $\beta: [0,\infty)\to\Real$ satisfying $$\dot{\beta} = (1 - \beta)\varphi(\beta),\;\beta(0)=0,$$ where $\varphi(s)>0$ for all $s\ge 0$, there exists a nonempty interval $I$ such that the function $G: I\to\Real$, defined by $s \mapsto (1 - s)\varphi(s)$, is monotonically decreasing on $I$. It can be verified that $\beta$ is strictly increasing and that $\beta(s)\ra 1$ as $s\ra\infty$. We further define 
\begin{equation*}
    W(x)=\begin{cases}
    \beta(V(x)),\;\;x\in\D\\
    1,\;\;\text{otherwise}.
\end{cases}
\end{equation*}

Then, by \cite[Chapter II, Proposition 2.5]{bardi1997optimal}, $W$  is the unique viscosity solution to  the Zubov equation \eqref{E: zubov_in}
in $\D$, with $W(0) = 0$. By the construction of $\beta$, $W(x)\ra 1$ as $x\ra\partial\D$. On the other hand, $W(x) = 1$ for all $x\in\Ds\setminus\D$ is a trivial solution to \eqref{E: zubov_in}. 
Following a similar procedure as in     \cite[Section 4]{camilli2008control}, $W$ is the unique viscosity solution to \eqref{E: zubov_in}. Additionally, by virtue of Remark \ref{rem: equiv}, such $W$ readily qualifies as a CLBF as defined in Section \ref{sec: Prob} via proximal gradients. Such $V$ and $W$, once obtained, can be used to design (not necessarily continuous) feedback controls using the same procedure as for CLFs \cite{sontag1989universal, sontag1999stability}.

\begin{rem}
    One can consider a special form of $\eta(x,u)=\theta\frac{x^{\mathsf{T}}Qx+u^{\mathsf{T}}Ru}{(1-h(x))^k}$  for  $\theta, k>0$ and positive definitematrices $Q$ and $R$, to ensure $\eta$ a valid proper indicator. In this case, one can let $\gamma(x) = \frac{(1-h(x))^k}{\theta}$, and set $\tilde{f}(x) = \gamma(x)f(x)$and $ \tilde{g}(x) = \gamma(x)g(x)$. One can verify that $(\tilde{f}, \tilde{g})$ and $(f, g)$ are topologically equivalent on $\D$ \cite{meng2025towards, camilli2008control}.  Then $V$  also solves for the corresponding HJB $$- \sup_{u\in\Real^m}\{DV(x) (\tilde{f}(x)+\tilde{g}(x)u)  -x^{\mathsf{T}}Qx-u^{\mathsf{T}}Ru\} = 0$$ and $W$ can be defined correspondingly (with the same regular cost $x^{\mathsf{T}}Qx-u^{\mathsf{T}}Ru\}$) as the solution of the Zubov equation associated with the system $(\tilde{f}, \tilde{g})$. Multiple-obstacle settings can also be handled using a similar idea as in \cite[Section IV]{meng2025towards}.\qed
\end{rem}

\subsection{Discussion}
In \cite{quartz2025converse}, the authors provide a succinct formula for  $$\hat{V}(x) = \frac{v(x)}{v(\phi(T(x);x))},$$ regardless of differentiability, constructed from a CLF $v:\D_0\ra\Real$ and the corresponding closed-loop control. Here, $T(x)$ denotes the first hitting time of the boundary of the safe set for the closed-loop system. One can show that such a $\hat{V}$ is equivalent to the construction of $V$ in this paper: both approaches quantify $\D$ by tuning the level sets of an existing CLF on $\D_0$, although in different ways. However, the given formula satisfies $$DV(x)(f(x)+g(x)u) = -\omega_1(x)$$ for some positive definite continuous function $\omega_1$ with $D\omega_1(x)\ra 0$,  but also requires a second PDE to characterize $T(x)$, which is impractical. In contrast, this work directly modifies the cost function in the associated optimal control problem by penalizing trajectories and controls that leave the safe set, leading to a single PDE whose solution serves as a CLBF, but at the cost of lacking a general explicit representation of the solution.


\section{Conclusion}

In this work, we developed a Zubov–HJB PDE framework for constructing control Lyapunov–barrier functions on the domain of null-controllability with safety constraints. The proposed formulation yields a maximal CLBF characterized as the viscosity solution of a single PDE, enabling feedback synthesis with stability and safety guarantees without requiring differentiability. The construction relies on a compatibility assumption that ensures well-posedness and desirable boundary behavior.  While this assumption is standard and no constructive counterexample is known, its necessity remains open. Future work will focus on weakening this condition by further investigating the interaction between system flows and the boundary of the safe set, as well as developing efficient numerical methods for solving the proposed PDE.

\bibliographystyle{plain}
\bibliography{ecc26}

\appendix
\section{A}
\begin{lem}\label{lem: safe_control}
Let $\mathcal{N}_\eps:= \{x\in S: 1-\eps\leq h(x)<1\}$. If for each $x\in \mathcal{N}_\eps$ there exists a $u\in\Real^m$ such that $\langle \xi, f(x)+g(x)u\rangle\leq \alpha(1-h(x))$ for all $\xi\in\partial_Ph(x)$, then $S$ is controlled invariant. 
\end{lem}
\begin{proof}
 For simply, we denote $x(t)$ of some controlled path under some $u\in\U$.    Fix $x_0\in\mathcal{N}_\eps$, and let $z(t)$ solve $\dot z =\alpha(1-z)$ with $z_0:=z(0)= h(x_0)$. It is clear that $z(t)<1$ for all $t\geq 0$. 
    Now consider the closed set $\text{epi}(h) = \{(x,z): h(x)\leq z\}$, which is   the epigraph of $h$. If we can show that $\text{epi}(h)$ is controlled invariant, then there exists a $u\in \U$ such that $h(x(t))\leq z(t)$ for all $t\geq $ for any $(x_0, z_0)\in \text{epi}(h)$, and $h(x(t))\leq z(t)<1$ for all $t\geq 0$. 

By definition, the proximal normal cone of $\text{epi}(h)$ is such that $$N_{\text{epi}(h)}^P(x,z)= \{\lambda(\xi, -1): \lambda\geq 0, \xi\in\partial_P h(x)\}.$$  Take any $\xi\in\partial_Ph(x_0)$. By assumption, there exists a $u\in\U$ such that $\langle \xi, f(x)+g(x)u\rangle\leq \alpha(1-h(x_0))$, which  implies $$\inf_u\langle (\xi, -1), (f(x)+g(x)u, \alpha(1-z_0))\rangle\leq 0.$$ Consequently,   $$\inf_u\langle \vartheta, (f(x)+g(x)u, \alpha(1-z_0))\rangle\leq 0,\;\vartheta\in N_{\text{epi}(h)}^P(x,z).$$ Then, by \cite[Chapter 4, Theorem 2.10]{clarke1998nonsmooth}, $\text{epi}(h)$ is controlled invariant. The proof is complete.
\end{proof}

\begin{lem}
    Under Assumption \ref{ass: comp},     $S$ is controlled invariant and asymptotic controllable. 
\end{lem}

\begin{proof}
Recall notation $\mathcal{N}_{\eps}$. Define $$\mathcal{I}_\eps:=\{x\in S: h(x)<1-\eps\}.$$ By Assumption \ref{ass: comp} and Lemma , for all $x\in\mathcal{N}$, there exists a $\kappa_1(x)\in\Gamma(x)$ such that $$\langle\zeta, f(x)+g(x)\kappa_1(x)\rangle\leq -w_\eps(x),\;\zeta\in\partial_P v(x)$$  For all $x\in\mathcal{I}_\eps\setminus\B(0;r)$, there exists a feedback $\kappa_2(x)$ such that $$\langle\zeta, f(x)+g(x)\kappa_2(x)\rangle\leq -w(x),\;\zeta\in\partial_P v(x)$$   Let $\tilde{w}(x) = w_\eps(x)$ for $x\in\mathcal{N}_\eps$ and $\tilde{w}(x) = w(x)$ for $x\in\mathcal{I}_\eps\setminus\B(0;r)$. Then it can be verified that $\tilde{w}(x)\geq \min\{\min_{x\in A_\delta \cap \mathcal{N}_\eps}w_\eps(x), \min_{x\in A_\delta \cap (\mathcal{I}_\eps\setminus \B(0;r))}w_\eps(x)\}=:\underline{w}>0$, where 
 $\delta\in(0, r)$ and $$A_\delta:=\{z\in S: v(z)\leq v(x)\}\cap \{x\in S: \|x\|\geq \delta\}$$ is compact.  By Clarke's weak decrease theorem applied to $v$, we have $v(x(t)) \leq v(x)- \underline{w}t$ for all $t\geq 0$ until entering $\B(0; r)$. The statement follows by combining this with Lemma \ref{lem: safe_control}.
\end{proof}

\end{document}